\documentclass[12pt]{article}
\usepackage{amssymb, amsthm}
\usepackage{pdfsync}
\usepackage[tbtags]{amsmath}
\newcommand{\Gm}{\Gamma}

\newcommand {\IN}{\mathbb N}  
\newcommand {\IR}{\mathbb R}   
\newcommand {\ZZ}{\mathbb Z}

\newcommand {\TT}{\mathbb T}

\newcommand {\D}{\mbox{${ D}$}}
\newcommand {\E}{\mbox{${ E}$}}

\newcommand{\Ekx}{E_\kappa^{x}}
\newcommand{\Ek}{E_\kappa}
\newcommand{\be}{\beta}

\newcommand{\zd}{\ZZ^d}

\newcommand{\Jt}{\frac1T\int_0^T\mu^{\otimes 2}_{\kappa,\beta,T}(X(t)\!=\!\widetilde{X}(t))dt}
\newcommand{\Ht}{\sum_{x\in{{\mathbb Z}^d}}\int_0^T dW_x(t)\delta_x(X(t))}
\newcommand{\Smx}{\sum_{x\in{{\mathbb Z}^d}}}
\newcommand{\Zkbx}{Z_{\kappa,\beta,t}(x)}
\newcommand{\Zkbt}{Z_{\kappa,\beta,t}}
\newcommand{\ZkbxT}{Z_{\kappa,\beta,T}(x)}
\newcommand{\ZkbT}{Z_{\kappa,\beta,T}}

\newcommand{\Lm}{\Lambda}
\newcommand{\Pski}{\Psi(\kappa,1)}

\newcommand{\JkbT}{J_{\kappa,\beta,T}}
\newcommand{\IkbT}{I_{\kappa,\beta,T}}

\newcommand{\Gkbr}{ \Gm(\be, r) }
\newcommand{\Pskb}{\Psi(\kappa,\be)}
\newcommand{\Pkx}{P_\kappa^x}
\newcommand{\fez}{\frac{e^{\beta H_T(X)}}{Z_{\kappa,\beta,T}}}
\newcommand{\ebT}{\exp\{\beta H_T(X)\}}
\newcommand{\ebt}{\exp\{\beta H_t(X)\}}

\newcommand{\iWx}{\int_0^T dW_x(t)\delta_x(X(t))}

\newcommand{\mthT}{\mu_{\kappa,\beta,T}}
\newcommand{\mkbs}{\mu_{\kappa,\beta,s}}
\newcommand{\mkxbT}{\mu^x_{\kappa,\beta,T}}
\newcommand{\mkbT}{\mu_{\kappa,\beta,T}}
\newcommand{\mkbt}{\mu_{\kappa,\beta,t}}
\newcommand{\Ikb}{\textup{I}_{\kappa,\beta}(r)}
\newcommand{\lt}{\left}
\newcommand{\It}{\frac1T\int_0^T\mu_{\kappa,\beta,t}^{\otimes 2}(X(t)\!=\!\widetilde{X}(t))dt}
\newcommand{\rt}{\right}

\newcommand {\8}{\infty}
\newcommand {\calD}{\mathcal D}
	\newcommand {\NC}{{\mathcal N}\!{\mathcal C}}
\newtheorem{stat}{Statement}
\begin{sloppypar}
   
\end{sloppypar}

\newtheorem{decth}[stat]{Theorem}
\newtheorem{prop}{Proposition}[section]
\newtheorem{cor}{Corollary}[section]
\newtheorem{conj}{Conjecture}[section]

\newtheorem{lemma}{Lemma}[section]
\newtheorem{remark}{Remark}[section]

\begin{document}

\title{ Overlaps and Pathwise  Localization
 in the Anderson Polymer Model}

\author{Francis Comets\footnote{Corresponding author. 
Universit\'e Diderot - Paris 7, Math\'ematiques, case 7012,
75205 Paris Cedex 13, France.  Partially supported by CNRS, UMR 7599.
\texttt{comets@math.univ-paris-diderot.fr}, \texttt{http://www.proba.jussieu.fr/$\sim$comets/.}}
\and 
Michael Cranston \footnote{Mathematics Dept., University of California, Irvine, CA. Research of second author supported by a grant from NSF, DMS 0854940.  }}

\maketitle

\begin{abstract}
We consider large time behavior of typical paths under the Anderson polymer measure. If $\Pkx$ is the measure induced by rate $\kappa,$ simple, symmetric random walk on $\ZZ^d$ started at $x,$ this measure is defined as
 \[d\mkxbT(X)={\ZkbxT}^{-1} \exp\left\{\beta\int_0^T dW_{X(s)}(s)\right\}d\Pkx(X)\]
where $\{W_x:x\in \ZZ^d\}$ is a field of $iid$ standard, one-dimensional Brownian motions,  $\beta>0, \kappa>0$ and
$\Zkbx$ the normalizing constant.
We establish that the polymer measure gives a macroscopic mass to a small neighborhood of a typical path as $T \to \8$, for parameter values outside the perturbative regime of the random walk, giving a pathwise 
approach to polymer localization, in contrast with existing results. The localization becomes complete as $\frac{\be^2}{\kappa}\to\8$ in the sense that the mass grows to 1.
The proof makes use of the overlap between two independent samples drawn under the Gibbs measure $\mkxbT$,
which can be estimated by the integration by parts formula for the Gaussian environment. 
Conditioning this measure on the number of jumps, we obtain a canonical measure which already shows scaling 
properties, thermodynamic limits, and decoupling of the parameters.  
\end{abstract}
\bigskip

\section{Introduction}
In this paper we consider a polymer model related to the parabolic Anderson equation. In particular, we give quantitative bounds on the overlap of the polymer measure in terms of an inverse temperature parameter. This gives a quantitative expression for the extent to which the polymer measure concentrates its weight near a particular path at low temperature.  
The Anderson polymer model is a measure on simple, symmetric, continuous-time random walks influenced by a random field, $\mathcal{W}=\{W_x:x\in\zd\},$ of $iid$ Brownian motions defined on some probability space $(\Omega, \mathcal{F}_t,\,Q)$ where the filtration is given by $\mathcal{F}_t=\sigma(\{W_x(s):\,0\le s\le t,\,x\in\zd\}).$ (Expectation with respect to $Q$ will be denoted by $E.$) This measure and related quantities have motivated a huge number of research papers
from many different perspectives. In the seminal reference \cite{CaMo94}  the model is used to give 
a mathematical account to intermittency, i.e. the existence of spots where most of the mass is concentrated. Those spots correspond to favorable configurations  of the field. The model
is also very appealing, with time-space {\it iid} environment  (the Brownian increments)
replaced by the configuration of an interacting particle system \cite{GadH06}, modeling 
a chemical reaction with moving catalysers. It has non trivial large deviation properties
\cite{CrGaMo10}, as a particular random growth model. 
 The one-dimensional totally asymmetric case, where the walker only jumps to the right, has a lower complexity than the symmetric case, as shown by the computations of annealed Lyapunov exponents \cite{BoCo12}; In this case an explicit solution was given in \cite{MoOc07}, with the 
 strongly asymmetric case as a small perturbation \cite{Moreno10}.
We consider the symmetric case and, like \cite{CaTiVi08},  \cite{CH}, \cite{RoTi05} and \cite{SepVa10}, our focus is to view the parabolic Anderson model as   a particular model of directed polymers in random medium, 
and a host of other references which may be found in \cite{CSY} for a general picture.

In order to describe this model,  start with the measure $\Pkx$ to be the measure on the canonical probability space, $\mathcal{D}([0,\infty),\zd)$ of right continuous paths which possess left limits everywhere, with a finite number of jumps of size one only on any finite interval. These are the typical sample paths of the simple symmetric rate $\kappa$ random walk. Here, as is usual, $\Pkx(X_0=x)=1$ and with respect to $\Pkx,$ the canonical process $X(t,\omega)=\omega(t),\,\omega\in\,\mathcal{D}_\infty=\mathcal{D}([0,\infty),\zd)$ is a Markov process with infinitesimal generator $\kappa\Delta$ where $\Delta$ is the discrete Laplacian defined by $\Delta f(x)=\frac{1}{2d}\sum_{||y-x||=1}(f(y)-f(x)).$ The  Anderson polymer model is the Gibbs measure on $\mathcal{D}_T=\mathcal{D}([0,T],\zd)$ defined by 
 \[\mkxbT(f)={\ZkbxT}^{-1}\Ekx\left[f \exp\left\{\beta\int_0^T dW_{X(s)}(s)\right\}\right].\]
for bounded measurable $f:\mathcal{D}_T\to {\IR}.$ The model has three parameters, the  inverse temperature
$\beta \in {\mathbb R}$ measuring the fluctuations of the environment, the polymer length $T$  and  the diffusivity $\kappa \in (0,\8)$ of the path
under the {\it a priori} measure $\Pkx$. The partition function $\ZkbxT$ is given by
\[\ZkbxT=\Ekx\left[\exp\left\{\beta\int_0^T dW_{X(s)}(s)\right\}\right].\] 
By the Feynman-Kac formula, 
\[u(t,x)=\Ekx\left[\exp\left\{\beta\int_0^t dW_{X(t-s)}(s)\right\}\right]\]
is the solution of the  time-dependent  parabolic Anderson equation (or stochastic heat equation)
\begin{eqnarray}\label{pam}
u(t,x)=1+\kappa\int_0^t\Delta u(s,x)ds+\beta\int_0^t u(s,x) \circ dW_x(s),
\end{eqnarray}
where $\circ dW$ denotes the Stratonovich differential of $W$.
In the space-continuous model, the logarithm of $u$ formally solves the Kardar-Parisi-Zhang equation
(KPZ), which is expected to be the scaling limit  of our model in a weak noise limit. In dimension $d=1$,
its distribution has been recently computed in \cite{AmCoQu11}, and a better understanding of the KPZ universality class is being achieved; see \cite{Corwin11} for a review.

%
%
The functions $u(t,x)$ and $\Zkbx$ thus have the same distribution and we will make use of the properties of $u(t,x)$ derived in \cite{CMS} and apply them to $\Zkbx.$
We shall use the notation $\Ek,\,\mkbT$ and $\ZkbT$ when $x=0.$ By spatial homogeneity of the field $\mathcal{W}$ we may confine ourselves to the study of the case $x=0.$ Our results concern the behavior of $\mkbT$ for ${\beta^2}/{\kappa}$ and $T$ large in arbitrary dimension. It will be shown that this measure concentrates its mass near a particular favorite path as ${\beta^2}/{\kappa}\to \infty.$ In a sense this means that there is a channel in the media in which most of the polymer paths reside. The establishment of the concentration is done by examining the overlap defined as
\begin{equation} \label{eq:defJ}
J_{\kappa,\beta,T}\equiv\Jt
\end{equation}
where $\mkbT^{\otimes 2}$ is the product measure of $\mthT$ with itself. The paths $X$ and $\widetilde{X}$ appearing in (\ref{eq:defJ}) are thus independent samples of paths drawn according to the measure $\mkbT.$
The quantity $J_{\kappa,\beta,T}$ is, for two independent samples sharing the same environment, the proportion of time  spent together. We also study another version of the overlap.
Define
\begin{equation} \label{eq:defI}
\IkbT=\It   \,.
\end{equation}
This version of the overlap measures the amount of time  up to $T$ that the endpoint of independent samples drawn with respect to the measure $\mkbt$ agree.

In statistical mechanics,  counterparts of these overlaps can be found, for the Sherrington-Kirkpatrick model and other ones for disordered systems.
Coming via integration by parts, the first overlap revealed most successful in the last decade
 (\cite{Gu01}, \cite{GuTo02}, \cite{Talabook1}, \cite{Talaparisi}) to study the low temperature regime.

Below, we shall see that the overlap in our polymer model arises in a similar fashion. By taking the logarithmic Malliavin derivative of the partition function with respect to $W_x(t)$ one arrives at $J_{\kappa,\beta,T}.$   By taking the logarithmic It\^{o} derivative of the partition function one arrives at $\IkbT.$ Both these overlaps take values in [0,1], they vanish in the limit of large $T$'s under the free measure,
i.e. for $\be=0$, and positivity of each one implies a specific form of localization. 
It is known that positivity of $\IkbT$ is equivalent to positivity of the difference between the annealed and the quenched free energies. Here, we show that positivity of $J_{\kappa,\beta,T}$
essentially amounts to this difference being strictly increasing as a function of $|\be|$.
Then, using the  logarithmic asymptotics 
of the partition function from \cite{CMS}, we find that both versions of the overlap converge to 1 as $\beta^2/\kappa
\to \8$, hence achieving its upper bound. This allows us to give a precise account of { path localization} by identifying
the {\bf favourite end-point},  and the {\bf favourite path} for the polymer. We introduce here  a sequence of measurable functions $y^*_T: [0,T] \to {\mathbb Z}^d$ such that the proportion of time when the polymer is {\it equal} to the favourite path is positive in the localized phase, and 
 converges to 1 as $\beta^2/\kappa\to \8$: 
\begin{equation} \label{eq:pathloc}
\liminf_{T \to \8} E \mkbT \left(
\frac{1}{T} \int_0^T\delta_0\big(X(t)-y^*_T(t)\big)dt \right) \longrightarrow 1 \quad {\rm as} \; \beta^2/\kappa
\to \8.
\end{equation}
The function $y^*_T$, defined in (\ref{def:favouritepath}),  depends on $\kappa, \beta$ and on the environment,
it is called the favourite path. Our statement improves on the literature on polymers by concerning the path itself, not only the terminal location at time $T$ as in \cite{CaHu02} and 
\cite{CSYo}. Also, we can obtain, in this model,  
complete localization in the limit, in the sense of (\ref{eq:pathloc})
where  the right hand side is equal to its maximum value 1.   

We now mention related results on complete localization.  Strong concentration for the  directed polymer in a random environment for parabolic Anderson model (space dependent only) with a Pareto potential was established in  \cite{KLMS}. The main difference is
that there, the favourable sites in the environment have a simple characterization in terms of the potential. In the time-discrete case with heavy tailed potentials (time-space dependent), see \cite{AuLo10} for similar conclusions. When the 
tails are less heavy, the favourite corridors can no longer be characterized by maxima of the potential, they depend in a much more subtle manner on the environment; Though they are not anymore explicit, site localization can still be proved \cite{Va07}. Note that in the discrete case, only little is known on
the random geodesics \cite{New} in first passage percolation, which are the zero-temperature favourite paths. In our parabolic Anderson model, the potential has strong decay, but we can prove strong localization. For the solution of the KPZ equation in one dimension, the distribution of the favourite end-point  has been recently computed in \cite{MoQuRe11}, it is the $\arg \max$
of an ${\rm Airy}_2$ process minus a parabola.

The parabolic Anderson model, compared to other directed polymers, has some nice scaling properties, it also
decouples some parameters due to its Poissonian structure. With a Gaussian potential, it allows a simple integration by parts 
formula, which comes from the Malliavin calculus. 

 The organization of this paper is as follows, in Section $2$ we present our main results. Section $3$ contains preliminaries from the Malliavin calculus giving a first version of the overlap. Section $4$ gives an Ito calculus derivation of the second version of the overlap, and contains the proofs for localization. In Section $5$ we prove the results for the distribution of the number of jumps under the polymer measure, and regularity properties of the favourite end-point and path 
 in Section \ref{sec:attributes}.

\section{Main Results}
From now on, we adopt the notation
\begin{eqnarray}
H_T(\gamma)=\int_0^T dW_{\gamma(s)}(s),\,\,\gamma\in\mathcal{D}_T.
\end{eqnarray}

\subsection{Thermodynamic limits}
Our first goal is to determine the large deviation rate for the number of jumps of paths with respect to the measure $\mthT.$ We begin with the exponential growth rate of $\ZkbT$ conditioned on the number of jumps of the process $X$ up to time $T$ which we denote by $N(T,X).$
\begin{prop} \label{th:existsGm}
For $r \geq 0$, the limit   
\begin{eqnarray}
  \label{def:Gm}
  \Gm( \be,r) &=& \lim_{T \to \8}
 T^{-1} \ln \Ek  \left[ \exp\{\be H_T(X)\}| N(T,X)=[rT] \right]
\end{eqnarray}
exists a.s. and in $ \textup{L}^p,\,\,p \in [1, \8).$ The limit is deterministic, symmetric and convex in $\beta,$ continuous in $(\be,r),$ independent of $\kappa$, and satisfies the scaling relation
  \begin{equation}
    \label{eq:scalingGmseul}
    \Gm(\be, r)= a \Gm(a^{-1/2}\be, a^{-1}r),\qquad a>0.
  \end{equation}
\end{prop}

	Let $ \textup{I}_\kappa$ be the Cram\'er transform (i.e. the large deviation rate function) of the Poisson distribution with parameter $\kappa$,
\begin{eqnarray}
  \label{def:Ikappa}
 \textup{I}_\kappa(r)= r \ln (r/\kappa) - r+\kappa,\qquad r \geq 0.
\end{eqnarray}
We will see that the function $r \mapsto \Gkbr -  \textup{I}_\kappa(r)$ is 
concave on $\IR_+$, and tends to $-\8$ as $r \to +\8$.
\begin{prop}[Free energy]\label{th:valfreeenergy}
 The limit 
$$
\Pskb = \lim_{T \to \8} T^{-1} \ln \ZkbT
$$
exists  a.s. and in $ \textup{L}^p,\,\,p \in [1, \8),$ and is equal to
\begin{eqnarray}
\Pskb= \sup \{ \Gkbr -  \textup{I}_\kappa(r);\,\, r \geq 0\}.
  \label{eq:valfreeenergy}
\end{eqnarray}  
\end{prop}
In particular, for all positive $a$,
  \begin{eqnarray}
\nonumber 
\Pskb&=& a \Psi( a^{-1}\kappa,a^{-1/2}\be)\\
  &=& \beta^2 \Psi(\beta^{-2} \kappa,1). \nonumber\\
    &=& \kappa \Psi(1, \kappa^{-1/2}\beta).     \label{eq:scalingPsiseul}
  \end{eqnarray}
From this we obtain the quenched large deviation rate function for the distribution of the number of  jumps of the polymer.
\begin{decth}[Large deviations]\label{th:ldp} Define $ \textup{I}_{ \kappa,\be}$ to be the convex function
$$ \Ikb= -\Gkbr +  \textup{I}_\kappa(r) + \Pskb.
$$
Then 

  \begin{equation} \label{eq:ldpp}
\lim_{T \to \8, n/T \to r} T^{-1} \ln \mkbT \big(N(T,X)=n\big) = 
-  \Ikb,\,\,a.s.. 
  \end{equation}
\vspace{.15in}

Moreover, for a.e. realization of the environment $\mathcal{W}=\{W_x(\cdot):\, x \in 
{\mathbb Z}^d\}$, and all subsets
$B \subset \IR_+$,

  \begin{eqnarray*} \label{eq:ldp}
- \inf_{B^o} \, \Ikb &\leq&
\liminf_{T \to \8} T^{-1} \ln  \mkbT \big(N(T,X)/T \in B\big)\\
&\leq&
\limsup_{T \to \8} T^{-1} \ln  \mkbT \big(N(T,X)/T \in B\big)\\
&\leq&- \inf_{\bar B}\, \Ikb.
  \end{eqnarray*}
\end{decth}
\begin{remark} \label{rem:21} \normalfont  (i) ${\bf{Annealed \,bound}}.$
By Jensen's inequality, it is readily checked that both 
$$ \Pskb \leq \beta^2/2 \quad {\rm and} \quad\;
\Gm(\be, r) \leq \beta^2/2$$
hold. These, in addition to (\ref{eq:valfreeenergy}), imply that
\begin{equation} \label{eq:qu=an}
\Pskb =\beta^2/2 \iff
\Gm(\be, \kappa) =\beta^2/2,
\end{equation}
and, in such a case, $\Ikb$ has a unique minimum at $r=\kappa$.

(ii) ${\bf{Weak \ versus \ strong \,disorder}}.$ 
From (\ref{ibp}), (\ref{eq:tard}) it follows that
\begin{equation} \label{eq:gamma_c}
\Pskb =\beta^2/2 \iff \beta^2/\kappa \leq \Upsilon_c\;, 
\end{equation}
for some critical value $\Upsilon_c \in [0, \8)$ depending only on the dimension; finiteness
of $\Upsilon_c$ can be seen from (\ref{lambda}) below. 

 In dimension $d \geq 3$,  it is known by second moment method 
\cite{CH} that $\ZkbT \exp\{-T\be^2/2\}$ converges to a positive limit, so that the equality holds in 
the left member of (\ref{eq:qu=an}) for $\beta^2/\kappa $ small.  Hence,  $\Upsilon_c>0$
in that case.

 In dimension $d=1, 2$, it has recently been proved \cite{Bertin12} that $\Upsilon_c=0$,
 by extending the techniques introduced in 
  \cite{Lac10} for discrete models.
\end{remark}

\subsection{Overlaps and phase transition}
Recall the definitions (\ref{eq:defI}) of $\IkbT$ and (\ref{eq:defJ})  of $J_{\kappa,\beta,T}$.
Even though the quantity $J_{\kappa,\beta,T}$ places more restriction on the paths, measuring the fraction of time together from $0$ to $T$ with respect to $\mkbT^{\otimes 2},$ for large $\beta^2/\kappa$ it is essentially the same size as $\IkbT.$ The advantage of $J_{\kappa,\beta,T}$
is that  it involves a single Gibbs measure and therefore contains pathwise information.
We will prove the following result in Sections \ref{sec:prel} and  \ref{sec:ito}. We remark  that (\ref{eq:valI}) was also proven in \cite{CH}. Also, discrete time versions of  (\ref{eq:defJ82}) were established in \cite{CaHu02} and in \cite{Lac10}.
\begin{prop} \label{th:It}(Overlaps)
(i) For all $\beta$ and $\kappa$ the limit 
\begin{equation}
\widetilde{I}_{\kappa,\beta,\infty}=\lim_{T\to\infty}\IkbT \label{eq:limex}
\end{equation}•
 exists almost surely and is nonrandom, and is equal to
 \begin{equation} \label{eq:valI}
\widetilde{I}_{\kappa,\beta,\infty}=
1- \frac{2 }{\beta^2} \Psi(\kappa,\beta).
\end{equation}
(ii) The limit
\begin{eqnarray} \label{eq:defJ8}
\widetilde{J}_{\kappa,\beta,\infty}=\lim_{T\to\infty}E\left[\JkbT\right]
\end{eqnarray}
exists for all $\kappa$ and all $\beta$ except for an at most countable set of values of $\beta^2/\kappa$, and 
\begin{equation}
\label{eq:defJ82}
\widetilde{J}_{\kappa,\beta,\infty} =1-\beta^{-1} \frac{\partial }{\partial \beta} \Psi(\kappa,\beta).
\end{equation}
\end{prop}
The first step for (i) is an It\^o calculation, and for (ii) it's an integration by parts formula, which relies on the Malliavin calculus.
\medskip

We say that a function $f: \IR \to \IR$ is non-constant around $b \in \IR$ if, in all neighborhoods of 
$b$ there is some $b'$ with $f(b')\neq f(b)$. Define the subsets of $\IR$,
\begin{equation*}
\NC_1=\left\{b : \frac{\beta^2}{2}-\Psi(1,\beta) \; {\rm is\ non\!-\!constant\ around\ } b\right\}, \quad
\NC_\kappa=\kappa^{1/2}\NC_1 .
\end{equation*}
By (\ref{eq:scalingPsiseul}), $\beta \mapsto \frac{\beta^2}{2}-\Psi(\kappa,\beta)$ is non-constant around
every point in $\NC_\kappa$, and only there. Note that, by (\ref{eq:defJ82}), this function is non-decreasing as $|\be|$ is increased. Also, if $\beta$ is not in the subdifferential of $\Psi(\kappa,\cdot)$  at $\beta$, then $\beta \in \NC_\kappa$.
 \begin{cor}  \label{cor2.1}
 (i) We have, with $\Upsilon_c$ from (\ref{eq:gamma_c}),
 \begin{equation*}
\widetilde{I}_{\kappa,\beta,\infty} > 0 \iff \beta^2/\kappa > \Upsilon_c .
\end{equation*}
(ii) Similarly, for all $\kappa >0$,
\begin{equation*}
\Upsilon_c=\inf\{ \beta^2/\kappa: \widetilde{J}_{\kappa,\beta,\infty}>0 \} . 
\end{equation*}
Moreover, if 
$\beta \in \NC_\kappa$, there exists a sequence $\beta_n$ converging to $\beta$ such that
$\widetilde{J}_{\kappa,\beta_n,\infty}>0$. Finally, if the derivative of  $\Psi(\kappa,\cdot)$
at $\beta$ exists and is different from $\beta$, then $\widetilde{J}_{\kappa,\beta,\infty}>0$. 
 \end{cor}
 
 We also have asymptotic estimates on the overlaps.
 
 \begin{cor} \label{th4}
As $\beta^2/\kappa \to\infty$, we have (i)
\begin{equation*}
\widetilde{I}_{\kappa,\beta,\infty} = 1-\frac{\alpha^2}{2\ln (\beta^2/\kappa)}(1+o(1)) ,
\end{equation*}
and (ii)
\begin{eqnarray} \label{eq:limJ8}
\widetilde{J}_{\kappa,\beta,\infty}= 1-{\mathcal O}\lt(\frac{1}{\ln (\beta^2/\kappa)}\rt) .
\end{eqnarray}
\end{cor}

\subsection{End-point and path localization}

For fixed $\kappa, \be$ we define the {\bf favourite end-point} $x^*(t)$
for the polymer at time $t$ by
\begin{equation}
  \label{def:favourite}
  x^*(t) = \arg\max \left\{ E_\kappa \left[ \exp \{ \beta H_t(X)\}\delta_x(X(t)) \right]
: x \in {\mathbb Z}^d \right\},
\end{equation}
taking the first argument $x$ in the lexicographic order in case of 
multiple maxima. By definition, $x^*(t)$ maximizes the distribution
$\mkbt(X(t)=x)$ of the end-point location. 
Noting that for $t\le T,$ 
\[E_\kappa[e^{\beta H_T(X)}]=\sum_{x\in{{\ZZ^d}}} E_\kappa[e^{\beta H_T(X)}\delta_x(X(t))]\]
is almost surely finite, we see that there is at least one maximizer.  We will see in Proposition \ref{prop:1} (i) that, in fact, 
the spatial rate of decay  is superexponential. 

We introduce  for each $T\geq 0$ a new object, the {\bf favourite path} with time horizon $T$, defined as
 \begin{equation}
  \label{def:favouritepath}
  y^*_T(t) = \arg\max \left\{ E_\kappa \left( \exp \{ \beta H_T(X)\}\delta_x(X(t)) \right)
; x \in {\mathbb Z}^d \right\}.
\end{equation}
At each time $t \in[0,T]$, it maximizes $\mkbT(X(t)=x)$; it does exist by a similar argument. The path $t\to  y^*_T(t) $ is a.s. piecewise constant on $[0,T].$ Indeed, if at time $t$ there is only one maximizer $x$ of  $f(x,t,T)=E_\kappa \left[ \exp \{ \beta H_t(X)\}\delta_x(X(t)) \right],$ then by continuity 
$x$ will the be the unique maximizer for times close to $t.$ In Proposition (\ref{prop:1}), $(iv)$ below, we show that  $t \to  f(x,t,T)$ is $C^1$ a.s.  (denote by f' the derivative in $t$.) Now, assume that for  $x \neq y,    f'(x,t,T)-f'(y,t,T)$ has a density. 
Then it is different from $0$ a.s., and so one maximizer will win over the other ones for times close to $t.$ We won't go into the details of the existence of the density, since we won't explicitly use the just stated claims. Otherwise, we know very little about the path $t\to y^*_T(t).$

We start with technical results.
Denote $f(x,t)=E_\kappa \left[ \exp \{ \beta H_t(X)\}\delta_x(X(t)) \right]$ and
$f(x,t,T)= E_\kappa \left[ \exp \{ \beta H_T(X)\} \delta_x(X(t)) \right]$ for short notations.

\begin{prop} \label{prop:1}
(i) For $a<\ln 2$ there exists $C_0=C_0(t,{\mathcal W})$ such that
$$
f(x,t) \leq C_0 \exp - \{ a |x| \ln |x| \},\qquad x \in \ZZ^d.
$$
(ii) For $a<\ln 2$ there exists $C_1=C_1(T,{\mathcal W})$ such that
$$
f(x,t,T) \leq C_1 \exp - \{ a |x| \ln |x| \},\qquad x \in \ZZ^d, t \leq T.
$$
(iii)  The function $t \mapsto f(x,t)$ is almost surely H\"older continuous of every order 
less than $1/2$.\\
(iv)  The function $t \mapsto f(x,t,T)$ is almost surely of ${\mathcal C}^1$ class on $(0,T)$. 
\end{prop}
We comment on some observations on the  favourite attributes $x^*$ and $ y^*_T$.
Both depend on $\kappa,\,\beta$ (also $T$ for the second one) and on the environment 
$\mathcal{W}.$ Both  have long jumps: $x^* \notin \mathcal{D}_\8$ a.s., and $y^*_T  \notin \mathcal{D}_T$ with positive probability for all $T$.
By time continuity in (iii-iv) and since for a measurable $F(\omega, x)$, the (smallest) maximizer is a measurable function of $\omega,$ both $(t, \omega)\mapsto y_T^*(t)$ and $(t,\omega) \mapsto x^*(t)$ are measurable functions from $[0,T] \times \Omega$ to ${{\ZZ^d}}.$
Observe that the functions $y^*_T$ and $x^*$ are equal at time $t=T$, but they are not related otherwise. In general, $x^*(T/2)$ is different from $y^*_T(T/2)$.
The end-point process is $\mathcal{F}$-adapted,  
with $\mathcal{F}=(\mathcal{F}_t)_{t \geq 0}, \mathcal{F}_t=\sigma\left(\{W_x(s):0\le s\le t,\,x\in {{\ZZ^d}}\}\right).$
The other one , being only $\mathcal{F}_T$-measurable,  is anticipating.   


Now, here is a fundamental difference between the two. 
The mapping $t \mapsto x^*(t)$ has oscillations at those times $t$ when there are many maximizers: in view of Proposition \ref{prop:1} (iii),
the set of jump times then looks locally like the set of zeros of Brownian motion.  
In contrast, from differentiability in Proposition \ref{prop:1} (iv), we see that $ t \mapsto y^*_T(t)$
has no oscillations. The favourite path is much smoother than the favourite end-point viewed
as a process.

\medskip

Coming to our main results, we prove that the polymer concentrates  on the favorite end-point and path in the strong disorder 
region (loosely speaking), and overwhelmingly as $\beta^2/\kappa \to\infty$.
\begin{decth} \label{th:localization} (favourite site and path)

(i) If $\beta^2/\kappa>\Upsilon_c$, there exists a constant $C=C(\beta^2/\kappa)>0$ such that
\begin{equation} \label{eq:endpointloc}
\liminf_{T \to \8}  \frac{1}{T} \int_0^T \mkbt( X(t) = x^*(t)) dt \geq C \quad a.s.
\end{equation}
On the contrary, for  $\beta^2/\kappa \leq \Upsilon_c$, the left-hand side of (\ref{eq:endpointloc}) converges to 0.
Finally, for all $\epsilon>0$,  for all $\beta^2/\kappa$ large enough, we have a.s.,
\begin{eqnarray}
  (1-\epsilon)\;  \frac{\alpha^2}{4\ln (\beta^2/\kappa)}
&\leq&  \liminf_{T\to\infty}
 \frac{1}{T} \int_0^T \mkbt( X(t) \neq x^*(t)) dt \nonumber \\ \nonumber
&\leq&  \limsup_{T\to\infty}\frac{1}{T} \int_0^T \mkbt( X(t) \neq x^*(t)) dt\\
&\leq& (2+\epsilon)\; \frac{\alpha^2}{4\ln (\beta^2/\kappa)} \;.\label{eq:complendpointloc} 
\end{eqnarray}
(ii) For all $\kappa>0$ and $\beta \in \NC_\kappa$, there exists a sequence $\beta_n \to \beta$ such that
\begin{equation} \label{eq:pathloc1}
\liminf_{T\to \8} E \mu_{\kappa,\beta_n,T}  \left( \frac{1}{T} \int_0^T \delta_0(X(t)-  y^*_T(t))dt\right) >0.
\end{equation}
On the contrary, for $\beta_0 \notin \NC_\kappa$, then $\lim_{T\to \8} E \mkbT  \left( \frac{1}{T} \int_0^T \delta_0(X(t)-  y^*_T(t))dt\right)  =0$ for all $\beta$ in a neighborhood of $\beta_0$.
Finally, for all $\epsilon>0$,  for all $\beta^2/\kappa$ large enough, we have a.s.,
\begin{eqnarray}
  (1-\epsilon)\;  \frac{\alpha^2}{4\ln (\beta^2/\kappa)}
&\leq &
\nonumber
1- \limsup_{T\to \8} E \mkbT  \left( \frac{1}{T} \int_0^T \delta_0(X(t)-  y^*_T(t))dt\right)  \\  \nonumber
&\leq& 1-  
\liminf_{T\to \8} E \mkbT  \left( \frac{1}{T} \int_0^T \delta_0(X(t)-  y^*_T(t))dt\right)  \\
 \label{eq:pathloc2}
& \leq& (2+\epsilon)\; \frac{\alpha^2}{4\ln (\beta^2/\kappa)} 
\end{eqnarray}

\end{decth}


In words, the proportion of time the polymer sticks to the favourite path is positive in the long run if and only if the 
difference between  annealed and quenched free energies is strictly increasing with $|\beta|$. Moreover, this proportion tends to 1 as the ratio $\beta^2/\kappa$ becomes large. 
In the case (\ref{eq:pathloc1}) we say that {\bf path localization} holds at $(\kappa, \beta)$.  
In the case of (\ref{eq:pathloc2}), precisely when  
 the overlap of the polymer with the favourite path tends to 1, we say that  {\bf complete path localization} takes place as $\beta^2/\kappa \to \8$.

Very little can be said so far on the favourite path $y^*_T$, which determines the corridor where most of the mass is concentrated. Both $x^*$ and $y^*_T$ are complicated functions of the environment. However, from the above theorem, it is approximated, in the distance ${\rm dist}_T(x,y)= T^{-1} \int_0^T \delta_0(x(t)-y(t))dt$, within order of $ 1/\ln^2(\beta^2/\kappa))$ accuracy, by a path $\gamma$ from $\mathcal{D}_T$, with jump density $T^{-1} N(T,\gamma) \in (r_{\rm max}(\kappa,\be) -\epsilon, r_{\rm max}(\kappa,\be) +\epsilon)$ by Theorem \ref{th:ldp}.
Since 
$$
\kappa^{-1}r_{\rm max}(\kappa,\be) \sim
 \frac{\alpha^2 (\be^2/\kappa)}{4 \ln^2(\be^2/\kappa)},
$$
the path $y^*_T$ becomes wilder as $\beta^2/\kappa$ increases, but within certain limits.

We end with a conjecture, which holds for polymer models on trees.
\begin{conj} \qquad 
\qquad   $ \NC_\kappa = \big\{ \beta: |\beta| \geq (\kappa \Upsilon_c)^{1/2}\big\}$. 
\end{conj}


\subsection{Asymptotics of Lyapunov exponents for parabolic Anderson model}
The existence of the quenched  Lyapunov exponent
$\Pski$ is well-known in parabolic Anderson model
literature, see for example  \cite{CaTiVi08}, \cite{CKM}, \cite{CaMo94} and \cite{CMS}. In \cite{CKM} and  \cite{CMS} it was shown that
\begin{eqnarray}\label{lambda}
\Pski\sim\frac{\alpha^2}{4\ln \frac1\kappa},\,\,\kappa\searrow 0,
\end{eqnarray}
where $\sim$ means that the ratio of the two sides tends to one.
A few words about the constant $\alpha$ are in order. Define the space of paths 
\[\Gamma_{[0,n],n}=\left\{\gamma\in \mathcal{D}_n:\,\gamma:[0,n]\to{\mathbb Z}^d,
\,N(\gamma,n)=n\right\}.\]
The superadditive functional 
\[A_{n}=\sup_{\gamma\in\Gamma_{[0,n],n}}H_n(\gamma)\]
is the supremum of a Gaussian field 
\[\left\{H_n(\gamma):{\gamma\in\Gamma_{[0,n],n}}\right\}\]
indexed by the set $\Gamma_{[0,n],n}.$ This set has a suitably bounded entropy, which  allows the conclusion, see \cite{CKM} and \cite{CMS},  that there is a finite, positive constant $\alpha$ such that
\[\lim_{n\to\infty}\frac1nA_{n}=\alpha,\,\,a.s.\]
This is the constant $\alpha$ appearing in (\ref{lambda}).
Thus, by the scaling relation (\ref{eq:scalingPsiseul}),
it follows that
\begin{eqnarray} \label{eq:huit}
\begin{split}
\Pskb 
\sim \frac{\alpha^2\beta^2}{4\ln(\beta^2/\kappa)},\,\,\, \beta^2/\kappa \to \infty
\end{split}
\end{eqnarray}
In particular, for $k>0$,
$$
\lim_{\beta \to \8} \Psi(\beta^2e^{-k\beta^2},\beta) = \frac{\alpha^2}{4k}.
$$
We give a streamlined approach to (\ref{lambda}) and  (\ref{eq:huit}).
Obviously, since $\textup{L}^p$ norms approach the $\textup{L}^\infty$ norm as $p\to\8,$
\begin{align*}
\lim_{\be \to \8}  \frac{1}{\beta T} 
\ln E_\kappa  \big[ \exp\{\be H_T(X)\}| N(T,X)=[rT] \big]
&=\sup_{\gamma\in\Gamma_{[0,T],[rT]}}H_T(\gamma)\\
&\equiv A_{T,r}
\end{align*}
By taking the limit as $T \to \8$ and 
interchanging the limits $T \to \8$ and $\be \to \8,$ a step to be justified later,
see (\ref{eq:stchely3}), we get 
\begin{equation} \label{eq:stchely2}
\lim_{\be \to \8}  \beta^{-1} \Gkbr = \lim_{T \to \8} 
T^{-1}A_{T,r}.
\end{equation}
Now, by the Brownian scaling, $ A_{T,r} \stackrel{\mathcal{L}}{=} \sqrt{r} A_T,$ making the 
previous limit equal to $\alpha \sqrt{r}$.
Let $r_{\rm max}(\kappa,\be)$ be the set of maximizers of the right-hand side
of (\ref{eq:valfreeenergy}). This set is a non-empty,
compact interval included in $\IR_+^*$, and we conjecture it reduces to a 
single point. Observe by scaling that
$$ r_{\rm max}(\kappa,\be)= \kappa r_{\rm max}( 1,\kappa^{-1/2}\be).$$
Parts $(i)$ and $(ii)$ of the following proposition were established in \cite{CKM}, \cite{CMS} and \cite{MRT08}.
\begin{prop} \label{th:CMS} \textup{(i)} Almost surely,
\begin{equation*}
\lim_{\be \to \8}  \beta^{-1} \Gkbr= \alpha \sqrt{r}
\end{equation*}
locally uniformly for  $r \in (0, +\8)$. \\
\textup{(ii)} As $\be^2/\kappa \to \8$, 
\begin{align*}
\Pskb \sim&\;
\kappa \sup \left\{ \alpha \be \sqrt{\frac{r}{\kappa}}
-\textup{I}_1(r) :\,r \geq 0\right\} \\
\sim&\; \frac{\alpha^2\beta^2}{4\ln(\beta^2/\kappa)}\;.
\end{align*}
\textup{(iii)} As $\be^2/\kappa \to \8$, 
\begin{equation*}
r_{\rm max}(\kappa,\be) \sim 
\kappa \times \frac{\alpha^2 (\be^2/\kappa)}{4 \ln^2(\be^2/\kappa)}\;.
\end{equation*}
\bigskip
\end{prop}
Recall $r_{\rm max}$ is defined as an interval, the last statement means that both endpoints are equivalent 
to the right-hand side.
The behavior in (iii) for the typical number of jumps ($r_{\rm max}(\kappa,\be) >\!> \kappa$) is drastically different from that in the weak disorder regime
in Remark \ref{rem:21}, (i). 
%
%
%
%
%
\section{Preliminaries from Malliavin calculus.} \label{sec:prel}
Express minus the Hamiltonian for a fixed path $X$ as
\[H_T(X)=\Smx\iWx.\]
This has the form 
\[W(h)=\Smx \int_0^T 
h(t,x)dW_x(t),\]
 where $h(t,x)$ depends on $X$ by the relation $h(t,x)=\delta_x(X(t)).$ Obviously, $h\in \textup{L}^2([0,T]\times{{\mathbb Z}^d}).$ The family $\{W(h): h \in  \textup{L}^2([0,T]\times{{\mathbb Z}^d})\}$ is called a centered, isonormal Gaussian family, and defines an abstract Wiener space as in  \cite{Nua06} or \cite{UstZa00}. The Malliavin derivative $DF$ of a square integrable random variable $F$ defined on this space is, when it exists, a random element of $ \textup{L}^2([0,T]\times{{\mathbb Z}^d})$, that we will view as a stochastic process $DF=(D_{t,x}F)_{t,x}$ indexed by time and space. The Malliavin derivative $D_{t,x}$  is heuristically equal to $\frac{\partial}{\partial (dW_x(t))}$ and can be formally computed as such.  
The Malliavin derivative of $H_T(X)$ is thus the element of $ \textup{L}^2([0,T]\times{{\mathbb Z}^d})$ defined by

\[D_{t,x}H_T(X)=\delta_x(X(t)).\]
Then taking $f(y)=e^y$ and applying the chain rule, we find the Malliavin derivative of $f(\beta H_T(X))$ is given by
\[D_{t,x}f(\beta H_T(X))=\beta f(\beta H_T(X))\delta_x(X(t)). \] 
Taking the average over paths $X$ and then differentiating yields
\[D_{t,x}\ZkbT=\beta E_\kappa[\delta_x(X(t))\ebT].\]
Note that we need to invoke not only linearity but also continuity to get this identity. 
Using again the chain rule, we obtain
\begin{eqnarray}
\begin{split}
D_{t,x}\ln \ZkbT=\beta\mkbT(\delta_x(X(t))).
\end{split}
\end{eqnarray}
The crucial point is that
\begin{align}
\begin{split}
  \label{eq:L2normderivative}
 \| D \ln \ZkbT\|^2_{L^2([0,T]\times{\mathbb Z}^d)}&=
\beta^2 \int_0^T\mkbT^{\otimes 2}(X(t)=\widetilde{X}(t))dt \\
&= \beta^2 T \, \JkbT\\
&\leq \beta^2 T.
\end{split}
\end{align}
The  integration by parts formula
\begin{equation}
\label{eq:ibpM}
E[W(h)F] = E \big[ \sum_x \int_0^T h \; D_{t,x} F \; dt  \big],
\end{equation}
is the central tool in Malliavin calculus, see \cite{Nua06}.

\begin{lemma}
\begin{equation}
\label{ibp}
\frac{\partial}{\partial \beta}E[\ln \ZkbT]
=\beta T\big[1-E[\JkbT]\big].
\end{equation}
\end{lemma}
\begin{proof} 
Differentiating inside the integral, we obtain
$$
\frac{\partial}{\partial \beta}E[\ln \ZkbT]= E[\mkbT(H_T(X))].
$$
Then, we write
\begin{eqnarray} \nonumber
\begin{split}
E\left[\mkbT(H_T(X))\right]=&E\left[\mkbT\left(\Ht\right)\right]\\
=&\Smx EE_\kappa\left[\fez \iWx\right]\\
=&\Smx \int_0^T E_\kappa\left[E\left[\fez dW_x(t)\right]\delta_x(X(t))\right]\\
=&\Smx \int_0^T E_\kappa\left[E\left[D_{t,x}\fez \right]\delta_x(X(t))\right]dt
\end{split}
\end{eqnarray}
from Gaussian integration by parts, see \cite{Nua06}. (A less pedestrian -- though equivalent --
computation is to apply directly the formula (\ref{eq:ibpM})). Then,
\begin{eqnarray*}
\begin{split}
&E\left[\mkbT(H_T(X))\right]\\
\qquad
&=\Smx \int_0^T E_\kappa\left[E\left[\left(\beta D_{t,x} H_T(X)\fez -\fez \frac{E_\kappa[D_{t,x} H_T(\tilde X) \, e^{\beta H_T(\tilde X}]}{\ZkbT}\right)\right]\delta_x(X(t))\right]dt\\
\qquad&=\beta\Smx \int_0^T E_\kappa\left[E\left[\left(\delta_x(X(t))\fez -\fez\frac{E_\kappa[\delta_x(\tilde X(t)) e^{\beta H_T(\tilde X)}]}{\ZkbT}\right)\right]\delta_x(X(t))\right]dt\\
\qquad&=\beta\Smx \int_0^T E\left[\mkbT(\delta_x(X(t))) -\mkbT(\delta_x(X(t)))^2 \right]dt\\
\qquad&=\beta E\left[\int_0^T \left(1-\mkbT^{\otimes 2}(X(t)=\widetilde{X}(t))\right)dt\right]\\
\qquad&=\beta T\left[1-E[\JkbT]\right].
\end{split}
\end{eqnarray*}
\end{proof}
We will use the concentration of measure phenomenon in our analysis. The use of Malliavin calculus for concentration appeared in 
\cite{RoTi05} in the study of polymers, and earlier in \cite{To98} in  the study of mean-field disordered systems.
\begin{lemma}[Concentration]
  \label{lem:concentration}
Let $A$ be a Borel subset of the path space with $P_x^\kappa(A)>0$, and let 
$$\ZkbT(A)= E_\kappa\left[ \exp\{\beta H_T(X)\}; A\right].$$ Then, for all $u>0$,
$$
Q\left( |\ln \ZkbT(A) - E \ln \ZkbT(A) | \geq u \right)
\leq 2 \exp\left\{ -\frac{u^2}{2 \beta^2 T}\right\}.
$$
\end{lemma}
\begin{proof} Of course,  $ \ZkbT(A)=\ZkbT$ when $A$ is the full space. Following the above computations for  $\ZkbT$, we see that the derivative is equal to
$$D_{t,x}\ln \ZkbT(A)=\beta\mkbT(\delta_x(X(t))|A),$$
and, as in (\ref{eq:L2normderivative}), its norm
$\textup{L}^2([0,T]\times {\ZZ}^d)$ is a.s. bounded by $ \beta^2 T$.
The lemma follows from Theorem B.8.1 in \cite{UstZa00}.
\end{proof}
Next observe from (\ref{ibp}) that
\begin{eqnarray} \nonumber
\Pskb&=&\lim_{T\to\infty}T^{-1}E\left[\ln\ZkbT
\right]\\ \label{eq:tard}
&=&\lim_{T\to\infty}\int_0^\beta r\left[1-E\left[J_{\kappa,r,T}\right]\right]dr.
\end{eqnarray}
Define 
$$\calD_1'=\{\beta \in \IR: \Psi(1,\beta) {\rm \ is\ differentiable\ at\  } \beta \} ,\quad \calD_1 =\calD_1'\setminus \{0\},$$
 and
\begin{eqnarray*}
 \calD_\kappa = \kappa^{1/2} \calD_1.
\end{eqnarray*}
By (\ref{eq:scalingPsiseul}), $\beta \mapsto \Psi(\kappa, \beta)$ is differentiable on $\calD_\kappa$ for all $\kappa$, and by convexity, and the complement of $\calD_\kappa$ is at most countable.
By standard convexity arguments, the derivative $\beta(1-E\left[J_{\kappa,\beta,T}\right])$ of 
$T^{-1}E\left[\ln \ZkbT \right]$ converges on $\calD_\kappa$ to $(\partial/\partial \beta)
\Pskb$.
Hence, the limit 
\begin{eqnarray} 
\label{eq:defJ82a}
\widetilde{J}_{\kappa,\beta,\infty}:=\lim_{T\to\infty}E\left[\JkbT\right] = 
1- \beta^{-1} \frac{\partial}{\partial \beta} \Pskb
\end{eqnarray}
exists for all $\kappa$ and  $\beta \in \calD_\kappa$, with
\begin{eqnarray} 
\label{eq:formint}
\Psi(\kappa, \beta)= \int_0^\beta r
\left[1-\widetilde{J}_{\kappa,r,\infty}\right] dr
\end{eqnarray}
and moreover
\begin{equation}\label{eq:decr}
\beta \mapsto \beta \big[ 1-\widetilde{J}_{\kappa,\beta,\infty}\big] {\rm \; is\ non \ decreasing.}
\end{equation}
\begin{prop} \label{th:Jt}
As $\beta^2/\kappa\to\infty,$
\begin{eqnarray*}
\frac{\alpha^2}{4\ln(\beta^2/\kappa)}(1+o(1))=&\frac{2}{\beta^2}\int_0^\beta r\left[1-\widetilde{J}_{\kappa,r,\infty}\right]dr,
\end{eqnarray*}
and 
\begin{eqnarray*}
\widetilde{J}_{\kappa,\beta,\infty}= 1-{\mathcal O}\left(\frac{1}{\ln(\beta^2/\kappa)}\right).
\end{eqnarray*}
\end{prop}
\begin{proof}
The first statement is straightforward from (\ref{eq:huit}) and (\ref{eq:formint}). By (\ref{eq:decr}), we have for $\beta  \geq 0$,
\begin{eqnarray*}
\int_0^\beta r\left[1-\widetilde{J}_{\kappa,r,\infty}\right]dr
 \ge&{\displaystyle \int_\frac{\beta}{2}^\beta} r\left[1-\widetilde{J}_{\kappa,r,\infty}\right]dr\\
\ge&\frac{\beta^2}{4}\left(1-\widetilde{J}_{\kappa,\beta/2,\infty}\right).
\end{eqnarray*}
Then, combining this with the first statement, we obtain
$$\widetilde{J}_{\kappa,\beta/2,\infty}\ge1-\frac{\alpha^2}{2\ln\lt(\frac{\beta^2}{\kappa}\rt)}(1+o(1)),
$$
yielding the same bound for $\widetilde{J}_{\kappa,\beta,\infty}$. This  completes the proof. 
\end{proof}
\section{An It\^o Calculation} \label{sec:ito}
In this section we use the It\^o calculus to obtain results on the overlap. We refer to \cite{CH} for the use of stochastic calculus in the study of the parabolic Anderson model. Recall
\[\Zkbt=\Ek\left[\exp\{\beta H_t(X) \}\right]\]
and that
\[\mkbt(f)=\frac{\Ek\left[f \exp\{\beta H_t(X) \}\right]}{\Zkbt}.\]
Note that
\[d\ln \Zkbt=\frac{1}{\Zkbt}d\Zkbt-\frac{1}{2\Zkbt^2}d\left<\Zkbt\right>.\]
where
\begin{eqnarray}
\begin{split}
d\Zkbt=&d\Ek\left[ \ebt\right]\\
=&\Ek\left[ d\,\exp\{\beta H_t(X) \}\right]\\
=&\Ek\left[
\beta \exp\{\beta H_t (X)\}\,dW_{X(t)}(t)+\frac{\beta^2}{2} 
\exp\{\beta H_t (X)\}\, dt\right]\\
\end{split}
\end{eqnarray}
and therefore,
\[d\left<\Zkbt\right>=\beta^2\Ek^{\otimes 2}\left[1_{X(t)=\widetilde{X}(t)}
\exp\{\beta H_t (X)\}\exp\{\beta H_t (\widetilde{X})\}
\right]\; dt\]
with $\tilde{X}$ an independent copy of $X$. 
Thus,
\[d\ln\,\Zkbt=\beta \mkbt(dW_{X(t)}(t))+\frac{\beta^2}{2}\left(1-\mkbt^{\otimes2}(X(t)=\widetilde{X}(t))\right)dt\]
and upon integration we get
\begin{equation}
\label{eq:1405}
\frac1t \,\ln\,\Zkbt=\frac1t M_t+  \frac{\beta^2}{2}\left(1-\frac{1}{t}\,\int_0^t\mkbs^{\otimes2}(X(s)=\widetilde{X}(s)) ds\right),        
\end{equation}•
where $M_t$ is a square-integrable martingale with quadratic variation 
$\left<M\right>$ given by
\[
\frac{d\left<M\right>_t}{dt}=\beta^2\Smx \mkbt(X(t)=x)^2=
\beta^2\mkbt^{\otimes2}(X(t)=\widetilde{X}(t))
.\]
As a consequence we
derive from (\ref{eq:1405}) both the existence of the limit (\ref{eq:limex}) and the relation (\ref{eq:valI}),
\[{\Pskb}=\frac{\beta^2}{2}\left(1-\lim_{t\to\infty}\,\frac1t\,\int_0^t\mkbs^{\otimes2}\left(X(s)=\widetilde{X}(s))\right)ds\right),\,\,a.s.\]
Indeed, 
\begin{eqnarray}
\begin{split}
\frac{d\left<M\right>_t}{dt}=&\beta^2\Smx \mkbt(X(t)=x)^2\\
\le&\beta^2\sup_{y\in{{\mathbb Z}^d}} \mkbt(X(t)=y)\Smx \mkbt(X(t)=x)\\
\le&\beta^2.
\end{split}
\end{eqnarray}
Thus, $\lim_{t\to\infty}\frac1t M_t=0$ and the rest is clear.

\begin{proof}[Proof of Proposition \ref{th:It}] We have just proved statements (i). By (\ref{eq:defJ82a}), and since the complement of $\D_\kappa$ is at most countable, we obtain the claims (ii).
\end{proof}
\begin{proof}[Proof of Corollary \ref{cor2.1}]  (i) is a consequence of  (\ref{eq:valI})  and of the definition (\ref{eq:gamma_c}) of $\Upsilon_c$. 

(ii) follows from the integral formula  (\ref{eq:formint}). For instance, for $0<\beta  < \beta_1$ such that 
$(\beta^2/2)-\Psi(\kappa,\beta) \neq (\beta_1^2/2)-\Psi(\kappa,\beta_1)$, the first term is smaller than the second one by
monotonicity, and the difference
$$
0<(\beta_1^2/2)-\Psi(\kappa,\beta_1)-(\beta^2/2)+\Psi(\kappa,\beta) =\int_\beta^{\beta_1} r \left[1-\widetilde{J}_{\kappa,r,\infty}\right] dr,
$$
so that there exists $r \in (\beta,\beta_1)$ such that $\widetilde{J}_{\kappa,r,\infty}>0$. From this and similar considerations, we easily  obtain the second statement. Under the assumption of the 
last claim, the equality (\ref{eq:defJ82}) holds, proving the claim.
\end{proof}
\begin{proof}[Proof of Proposition \ref{th4}] The claim  (i) is a consequence of (\ref{eq:huit}) and  (\ref{eq:valI}).The claim in (ii) is simply the second part of Proposition \ref{th:Jt}.
\end{proof}
\begin{proof}[Proof of Theorem \ref{th:localization}] 
Letting $a_T=1-\IkbT$ for a short notation, we have
  \begin{eqnarray*}
 \label{at:eqn}
a_T &:=& 1- 
\frac{1}{T} \int_0^T \mkbt^{\otimes 2}( X(t) = \widetilde X(t) ) dt \\
 &=& \frac{1}{T} \int_0^T \sum_{x \in {\mathbb Z}^d}
    \mkbt( X(t)=x )  \mkbt( X(t)\neq x ) dt\\
 & = &
 \frac{1}{T} \int_0^T \left( 1 - \sum_{x \in {\mathbb Z}^d}
    \mkbt( X(t)=x )^2 \right) dt.
 \end{eqnarray*}
Let us denote 
$$m_t= \mkbt( X(t) \neq x^*(t)), \qquad 
b_T= \frac{1}{T} \int_0^T 
m_t dt.
$$
By splitting off the term for $x=x^*(t)$ from the sum of the terms for $x\not=x^*(t)$ in (\ref{at:eqn}) we get,
\begin{eqnarray*}
  a_T
 & = &
 \frac{1}{T} \int_0^T \left( 1 - (1-m_t)^2 - \sum_{x \neq x^*(t)}
    \mkbt( X(t)=x )^2 \right) dt\\
 & = &
2 b_T - \frac{1}{T} \int_0^T m_t^2 dt 
- \frac{1}{T} \int_0^T  \sum_{x \neq x^*(t)}
    \mkbt( X(t)=x )^2  dt.
\end{eqnarray*}
Clearly, it implies that $a_T  \leq 2b_T$, but also an estimate in the reverse direction.
Since 
\begin{eqnarray*}
 \sum_{x \neq x^*(t)}    \mkbt( X(t)=x )^2 
 &\leq&    \sum_{x \neq x^*(t)}  \mkbt( X(t)=x^*(t) ) \times \mkbt( X(t)=x )\\
 &=& (1-m_t) m_t,
\end{eqnarray*}
by definition of $x^*$, we obtain 
$$b_T \leq a_T  \leq 2b_T.$$
Therefore, (i) in Theorem \ref{th:localization} follows from points (i) of Corollary \ref{cor2.1} and of Proposition\ref{th4}.
%
%
%

We now turn to the proof of (ii).
Repeating the same steps with $\mkbT$ instead of $\mkbt$, except for the splitting according to $x=y^*_T(t)$ or not,
we see that
$$
\bar a_T=1- 
E \frac{1}{T} \int_0^T  \mkbT^{\otimes 2}( X(t) = \widetilde X(t) ) dt, \quad
$$
and
$$
\bar b_T= E \frac{1}{T} \int_0^T  \mkbT( X(t) \neq y^*_T(t))
 dt,
$$ 
are such that
$$
\bar b_T \leq \bar a_T  \leq 2 \bar b_T .
$$
Since $ \bar a_T = 1-E \JkbT$, point (ii) in Theorem \ref{th:localization} follows from points
(ii) of Corollary \ref{cor2.1} and of Proposition \ref{th4}.
\end{proof}

\section{Jump Distribution} \label{sec:jumps}

Let $E_S$ denote the expectation for the simple, discrete time random walk 
$S=\{S(i); i \in \IN\}$ on $\ZZ^d$ with discrete time, and,  
for $n \in \IN$, 
$\TT_{n,T}=\{ (t_1,\ldots, t_n): 0<t_1<\ldots t_n<T\}$.
The quantity 
\begin{equation} \label{eq:cabourg1}
E_\kappa \left[ e^{\be H_T(X)} \vert N(T,X)=n \right] 
\end{equation}
does not depend on $\kappa$ and is equal to
\begin{eqnarray} \nonumber
\left[ \frac{(\kappa T)^{n}}{n!}e^{-\kappa T}\right]^{-1}&
{\mathop{\int\! \ldots\! \int}_{\TT_{n,T}}} 
E_S
\left[ 
\prod_{i=0}^n
e^{\be [W_{S(i)}(t_{i+1}) -W_{S(i)}(t_{i})]} 
\kappa e^{-\kappa (t_{i+1}-t_i)}
\right]  
dt_1\ldots
dt_n \qquad\qquad \qquad
\\ \nonumber = &
 T^{-n}n!
\mathop{\int \!\ldots\! \int}_{\TT_{n,T}}
E_S
\left[ e^{\be \sum_{i=0}^n [W_{S(i)}(t_{i+1}) -W_{S(i)}(t_{i})]}\right]  dt_1\ldots
dt_n, \qquad \qquad \qquad \qquad
\end{eqnarray}
where we have set $t_0=0, t_{n+1}=T$.
Under the law on path space defined by (\ref{eq:cabourg1}), the jump times
and jump values are independent, with respective distributions, 
uniform on $[0,T]$, and $P_S$.  

\begin{prop} \label{prop:existsLm}
The following limits exist a.s. and in $L^p, p \in [1,\8)$, and are equal:
\begin{eqnarray}
  \label{eq:Lm}
\Lm(\kappa, \be, r) &=& \lim_{T \to \8}
 T^{-1} \ln E_\kappa  \left[ e^{\be H_T(X)}; N(T,X)=[rT] \right]\\
& =& \lim_{T \to \8, n/T \to r}
 T^{-1} \ln E_\kappa  \left[ e^{\be H_T(X)}; N(T,X)=n \right]  
\nonumber
\end{eqnarray}
The limit is deterministic, convex in $\be$, and concave in $r$. 
\end{prop}

We will use the following observation, which has been found to be useful in similar 
situations, where the subadditive (or superadditive)
ergodic theorem does not apply.
\begin{lemma}[Stochastic superadditive lemma] \label{lem:stosup}
Let $U_t$ be an integrable random process indexed by $t$ in $\IN$ or $\IR_+$,
such that\\
\textup{(i)} $\E U_{t+s} \geq \E U_t + \E U_s, \quad s, t \geq 0,$\\
and\\
\textup{(ii)} $\frac{1}{t}(U_t- \E U_t)  \stackrel{*}{\longrightarrow} 0$ as $ t \to \8$,
where $ \stackrel{*}{\longrightarrow}$ 
is some stochastic mode of convergence (a.s., in probability, 
in $L^p$, \ldots). 

Then, as $ t \to \8$,
$$
\frac{U_t}{t} \stackrel{*}{\longrightarrow} \sup\left\{\frac{\E U_t}{t}; \,\,t \geq 0\right\}
$$
with the same mode of convergence.
\end{lemma}
\begin{proof}[Proof of Lemma \ref{lem:stosup}]
From (i) and the superadditive lemma, 
$$
\lim_{t \to \8} \frac{\E U_t}{t} = \sup\left\{\frac{\E U_t}{t}; \,\,t \geq 0 \right\}.
$$
The claims now follows from (ii).
\end{proof}
\begin{proof}[Proof of Proposition \ref{prop:existsLm}]
For $r \in {\mathbb Q}_+$, we now check that 
Lemma \ref{lem:stosup} applies to the sequence 
$$
U_t =  \ln E_\kappa  \left[ e^{\be H_t(X)}; N(t,X)=rt \right],
$$
$t \in \IN$ with $rt \in \IN$. 
First, define a probability measure on the set of paths $\mathcal{D}_\infty$ by $\mkbt^{(r)}(\cdot)=\mkbt(\cdot \,| N(t,X)=rt),$
we have for $s \in r^{-1}\IN$,
\begin{eqnarray} \nonumber
U_{t+s} &\geq&
\ln E_\kappa  \left[ e^{\be H_{t+s}(X)}; N(t+s,X)=r(t+s),  N(t,X)=rt \right]\\
\nonumber
&=& 
\ln E_\kappa  \left[ e^{\be H_{t}(X)}e^{\be [H_{t+s}(X)-H_{t}(X)]};
 N(t+s,X)=r(t+s),  N(t,X)=rt \right]\\ \nonumber
&\stackrel{\rm Markov}{=}& 
 U_{t} +
\ln \mkbt^{(r)}\left[ \exp U_{s} \circ \theta_{t,X(t)} \right] \qquad
(\theta \mbox{ time-space shift})\\ \label{eq:cabourg2}
&\stackrel{\rm Jensen}{\geq}& 
 U_{t} + \mkbt^{(r)}\left[U_{s} \circ \theta_{t,X(t)}\right]
\end{eqnarray}
Since $\{W_x:x\in\zd\}$ are independent Brownian motions, 
\begin{eqnarray*}
E \left[ U_{t} + \mkbt^{(r)}\left[U_{s} \circ \theta_{t,X(t)}\right]\right]
&=&
E U_{t} 
+E \mkbt^{(r)}\left[
E\left( U_{s} \circ \theta_{t,X(t)}\right)\right]\\
&=&
E U_{t}  +E U_{s},
\end{eqnarray*}
which, together with (\ref{eq:cabourg2}), proves  (i). 

To show (ii) with a.s. convergence, we combine concentration 
and martingale inequalities. Let $1/2 < a <1$, and $\{T_n\}_{n\ge1}$ be the
sequence defined by $T_1=1$, $T_{n+1}=T_n+T_n^a$. Then $T_n = 
n^{\frac{1}{1-a} + o(1)}$
 as $n \to \8$. By Lemma \ref{lem:concentration} with 
$A=\{N(t,X)=rt\}$, it is easily seen that (ii) holds along the sequence $T_n$ using the Borel-Cantelli lemma. We now bridge the gaps.
By It\^o's formula, 
$$
U_T= M_T - (1/2) \langle M \rangle_T + \beta^2T/2
$$
for some continuous martingale $M$ with $(d/dt)\langle M \rangle_t \leq \beta^2$ for all $t \geq 0$.
Fix a sequence $\epsilon_n \to 0$ with $\epsilon_n^{-1} = n^{o(1)}$. 
For $n$ large, $\beta^2 T_n^a < \epsilon_n T_{n+1}$, and then
\begin{eqnarray*}
  Q \left[ \sup_{T_n \leq T \leq T_{n\!+\!1}} 
|U_T\!-\!U_{T_n}\!-\!E U_T\!+\!EU_{T_n}|
> 2 \epsilon_n T_{n\!+\!1} \right] 
&\leq& 
  E \left[ \sup_{T_n \leq T \leq T_{n\!+\!1}} 
|M_{ T}\!-\!M_{T_n}|  >\epsilon_n T_{n+1}
\right]  \\
&
\stackrel{\rm Doob}{\leq}& 
(\epsilon_n T_{n+1})^{-2} 
E\left[ \langle M  \rangle_{ T_{n+1}}- \langle
M \rangle_{T_n} \right] 
\\ 
&\leq& 
(\epsilon_n T_{n+1})^{-2} 
\beta^2 (T_{n+1}-T_n),
\end{eqnarray*}
which defines a summable series if we choose $a \in (1/2,1)$ large enough.
By Borel-Cantelli, this completes the proof of (ii). 

The limit $\Lm$ is convex in $\be$ as a limit of convex functions.
We now check concavity in $r$.  First note that
$$
V_{T,n}=\ln E_\kappa  \left[ e^{\be H_{T}(X)};N(T,X)=n \right],
$$
which is equal to $U_T$ with $r=T/n$, satisfies
\begin{eqnarray}
\begin{split}\label{concavity}
V_{T+T',n+n'} 
&\geq& 
\ln E_\kappa  \left[ e^{\be H_{T+T'}(X)};N(T,X)=n,  N(T+T',X)=n+n' \right]\\
&=&
V_{T,n}+ \ln \mkbT^{(r)}
[\exp V_{T+T',n+n'}\circ \theta_{T,X(T)}]  \qquad (r=T/n).  
\end{split}
\end{eqnarray}
Proceeding  as in (\ref{eq:cabourg2}), and 
letting $T,T',n,n' \to \8$ in such a way that $n/T \to r, n'/T'\to r',
T/T' \to \lambda/(1-\lambda)$ with $r,r' \geq 0$ and $\lambda \in [0,1]$,
 we get
$$
 \Lm(\kappa, \be, \lambda r+(1-\lambda)  r'\big) \geq
\lambda \Lm(\kappa, \be, r)+(1-\lambda) \Lm\big(\kappa, \be, r'),
$$
i.e, the desired concavity.
\end{proof}
\begin{proof}[Proof of Proposition \ref{th:existsGm}]
Writing the conditional expectation as a ratio, we now see that the limit
$$ 
 \lim_{T \to \8}
 T^{-1} \ln E_\kappa  \left[ e^{\be H_T(X)}| N(T,X)=[rT] \right]
$$
exists, and is equal to $\Lm(\kappa, \be, r) + I_\kappa(r)$ by Cram\'er's 
theorem.  This shows the existence of $\Gm$, and also that
\begin{equation}
  \label{eq:formulaLm}
  \Lm(\kappa, \be, r)=\Gkbr - I_\kappa(r)\;.
\end{equation}

We now turn to the scaling relation. 
Under $P_\kappa(\cdot\,| N(T,X)=n )$, $X_{[0,T]}:=(X_t; t \in [0,T])$ 
has $n$ jumps on $[0,T]$
the values and times of which are independent, uniformly distributed. 
Then, with $X^{(a)}: s \mapsto  X(s/a)$, the following image laws
are equal
  \begin{equation}
    \label{eq:cabourg10}
X^{(a)}_{[0,aT]} \circ P_x^\kappa(\cdot\,| N(T,X)=n )
= X_{[0,aT]} \circ P_x^\kappa(\cdot\,| N(aT,X)=n ).
  \end{equation}
Also, $W^{(a)}_x(s)= a^{1/2} W_x(s/a)$ defines a collection, $\mathcal{W}^{(a)}=\{W^{(a)}_x:x\in\ZZ^d\},$
of independent
standard Brownian motions. Denoting by $t_i, \,i=1,\ldots n,$ the jump times
of $X$ and $t_0=0, t_{n+1}=T$, we have 
by definition, 
$$
H_T^{W}(X)
= 
\sum_{i=0}^n [W_{X_{t_i}}(t_{i+1})-W_{X_{t_i}}(t_i)]\\
= 
a^{-1/2}  H_{aT}^{W^{(a)}}(X^{(a)}),
$$
and also,
\begin{eqnarray}
\begin{split}\label{gammascale}
\frac{1}{T} \ln   E_\kappa \left[ e^{\be H_T(X)} \vert N(T,X)=n \right] 
&=&
\frac{a}{aT} \ln   
E_\kappa \left[ e^{{\be}{a^{-1/2}} H_{aT}^{W^{(a)}}(X^{(a)})
} \vert N(T,X)=n \right]\\
&=&
\frac{a}{aT} \ln   
E_\kappa \left[ e^{{\be}{a^{-1/2}} H_{aT}^{W^{(a)}}(X)
} \vert N(aT,X)=n \right]
\end{split}
\end{eqnarray}
by (\ref{eq:cabourg10}). The first scaling relation follows from taking the 
limit $T \to \8, n/T \to r$. The convexity in $\beta$ follows from H\"{o}lder's inequality. For continuity we need to establish $\lim_{a\to1}\Gamma(\beta,ar)=\Gkbr.$ But, by scaling, $\Gamma(\beta,ar)=a\Gamma(\beta/\sqrt{a},r)$ and the result follows from continuity in $\beta.$
\end{proof}
Here is a direct consequence of (\ref{eq:formulaLm}) and of the scaling.
\begin{cor} For all $r$, $\Lm$ is jointly convex in $(\kappa, \beta)$, and 
  \begin{equation}
    \label{eq:scalingLm}
\Lm(\kappa, \be, r) = a \Lm(a^{-1}\kappa, a^{-1/2}\be, a^{-1}r).
  \end{equation}
\end{cor}
\begin{proof}[Proof of Proposition \ref{th:valfreeenergy}]
For $a, b>0$, write
  \begin{eqnarray}
  \begin{split}
E_\kappa \left[ e^{\be H_T(X)} \right] =\sum_{aT \leq n \leq bT} &E_\kappa \left[ e^{\be H_T(X)} ; N(T,X)=n \right] +
E_\kappa \left[ e^{\be H_T(X)} ; N(T,X)> bT \right]\\
\quad+&E_\kappa \left[ e^{\be H_T(X)} ; N(T,X)< aT \right].
  \end{split}
  \end{eqnarray}
By the concentration inequality in Lemma \ref{lem:concentration}, as $T \to \8$,
$$
T^{-1}\left( \ln E_\kappa \left[ e^{\be H_T(X)} ; N(T,X)> aT \right]
- E\left[  \ln E_\kappa \left[ e^{\be H_T(X)} ; N(T,X)> aT \right] \right]\right)
\to 0
$$
a.s. and in $ \textup{L}^p$. By Jensen's inequality,
$$
T^{-1} E\lt[  \ln E_\kappa \left[ e^{\be H_T(X)} ; N(T,X)> bT \right] \rt]
\leq \beta^2/2 - \textup{I}_\kappa(B),
$$
which can be made arbitrarily negative by taking $b$ large.
Similarly, 
$$
T^{-1} E\lt[  \ln E_\kappa \left[ e^{\be H_T(X)} ; N(T,X)<aT \right] \rt]
\leq \beta^2/2 - \textup{I}_\kappa(a),
$$
which can be made arbitrarily negative by taking $a$ small.
Thus, for $a$ sufficiently small and $b$ sufficiently large,
$$\lim_{T\to\8}T^{-1} \ln E_\kappa \left[ e^{\be H_T(X)} \right]= \lim_{T\to\8}T^{-1} \ln E_\kappa \left[ e^{\be H_T(X)} ; aT<N(T,X) <bT \right].$$
Define $\Gamma_T(\beta,r)=T^{-1}E \ln E_\kappa \left[ e^{\be H_T(X)}|N(X,T)=[rT] \right]$ which is a convex function of $\beta$ converging point-wise to $\Gkbr$ which is also convex in $\beta.$  
The conditional version of the concentration inequality holds due to cancellation, that is
\[P\lt(\lt|\frac{1}{T}\ln E_\kappa \left[ e^{\be H_T(X)} | N(T,X)=[rT] \right] -\Gamma_T(\beta,r)\rt| \ge u\rt)\le 2\exp\lt\{-\frac{u^2T}{2\beta^2}\rt\}. \]
Write $r=sa$ with $s$ ranging over $C\equiv[1,\frac{b}{a}]\cap n^{-1}{\ZZ}.$ Since the number of points in $C$ grows like $n,$ we conclude by Borel-Cantelli that
\[P\lt(\sup_{s\in C}\lt|\frac{1}{n}\ln E_\kappa \left[ e^{\be H_n(X)} | N(n,X)=[sa T] \right] -\Gamma_n(\beta,sA)\rt|>\epsilon/4\,\,i.o.\rt)=0.\] 

We also note that by convexity of both $\Gamma_T$ and $\Gamma$ that $\Gamma_T(\beta,r)$ converges uniformly for $\beta$ in a compact interval to $\Gamma(\beta,r).$ By (\ref{gammascale}) it follows that $\Gamma_T(\beta,ar)=a\Gamma_{aT}(a^{-1/2}\beta,r)$ which implies uniform convergence of $\Gamma_T(\beta,r)$ to $\Gamma(\beta,r)$ for $r$ in a compact interval.

Finally, for $0<a<b<\infty,$
\begin{eqnarray*}
\sum_{r\in[a,b]\cap n^{-1}{\bf{Z}}} E_\kappa \left[ e^{\be H_n(X)} ; N(n,X)\!=\![rn] \right] &=&\sum_{r\in[a,b]\cap n^{-1}{\bf{Z}}} \Ek  \left[ \exp\{\be H_n(X)\}| N(n,X)=[rn] \right]\\
&&\quad\times P_\kappa(N(n,X)=[rn] )\\
&=& e^{o(n)}\sum_{r\in[a,b]\cap n^{-1}{\bf{Z}}} e^{\Gamma_n(\be,r)n}P_\kappa(N(n,X)=[rn] )\\
&=& e^{o(n)}\sum_{r\in[a,b]\cap n^{-1}{\bf{Z}}} e^{\Gamma(\be,r)n}P_\kappa(N(n,X)=[rn] )\\
&=& e^{o(n)} \sum_{r\in[a,b]\cap n^{-1}{\bf{Z}}} e^{(\Gamma(\be,r)-I_\kappa(r))n}.
\end{eqnarray*}
Thus,
\begin{eqnarray}
\begin{split}
T^{-1}\ln E_\kappa \left[ e^{\be H_T(X)} \right]\sim &\quad  T^{-1} \ln E_\kappa \left[ e^{\be H_T(X)} ; aT<N(T,X)< bT \right]\\
=&\quad  o(1)+T^{-1}\ln \sum_{r\in[a,b]\cap T^{-1}{\bf{Z}}} e^{(\Gkbr-I_\kappa(r))T}\\
=& \quad \sup\big\{\Gamma(\be,r)-I_\kappa(r); r \in [a,b]\big\}+o(1)\\
\to& \quad \Psi(\kappa,\beta),\qquad T\to\infty,
\end{split}
\end{eqnarray}
 by the standard Laplace method.
\end{proof}
\begin{remark}\normalfont
We have shown that
$$
\Pskb = \sup\{ \Lm(\kappa, \beta,r); r \geq 0\}.
$$  
Since $\textup{I}_\kappa$ is non negative and zero if $r=\kappa$, and by \textup{(\ref{eq:formulaLm}\textup)}, we have
$$
\Gm(\beta, r) = \sup\{ \Lm(\kappa, \beta,r); \kappa > 0\}.
$$
\end{remark}
\begin{proof}[Proof of Theorem \ref{th:ldp}] The first statement directly 
follows from the previous results. The large deviation principle 
(\ref{eq:ldp}) is proved in a manner similar to  Proposition \ref{th:valfreeenergy}
\end{proof}
\begin{proof}[Proof of Proposition \ref{th:CMS}.] 
By (\ref{eq:scalingGmseul}), 
\begin{equation} \label{eq:stchely1}
\frac{\Gm(\be,r)}{\be}=r\; \frac{\Gm(r^{-1/2}\be,1)}{\be} =
\sqrt{r} \; \frac{\Gm(r^{-1/2}\be,1)}{r^{-1/2}\be}.
\end{equation}
We claim 
\begin{equation} \label{eq:stchely3}
\lim_{\be \to \8}  \beta^{-1} \Gm( \beta, 1) = \alpha.
\end{equation}
i.e., (\ref{eq:stchely2}) holds for $r=1$, and that for all $\beta$,
$ \beta^{-1} \Gm( \beta, 1) \leq \alpha$. 
Indeed,
\begin{eqnarray}
\begin{split}\label{eq:gammaub}
 \frac{1}{\beta} \Gm( \beta, 1) =&\lim_{T \to \8}   \frac{1}{\beta T}\ln \Ek[e^{\be H_T(x)}|N(T,X)=T]\\
\le& \lim_{T \to \8}   \frac{1}{\beta T} \be A_T\\
=&\alpha.
\end{split}
\end{eqnarray}
Following the developments in \cite{CMS},  given $\epsilon>0,$ one can find a path $\gamma\in \mathcal{D}_T$ with $N(T,\gamma)=T$ and the jump times of $\gamma$ are separated by $\epsilon$ and $H_T(\gamma)>(\alpha-\delta)T$ where $\delta\to0$ as $\epsilon\to0.$ Moreover, writing $\eta\sim\gamma$ to mean that the jump times of $\eta$ are within $\epsilon/3$ of the jump times of $\gamma$ and the two paths jump to the same sites at these jump times we have
\[P(X\sim\gamma|N(T,X)=T)\ge \frac{T!}{T^T}\lt(\frac{\epsilon}{6}\rt)^T\sim \sqrt{T}\lt(\frac{\epsilon}{6e}\rt)^T.\]
In addition, for $X\sim \gamma$ one has eventually, $H_T(X)\ge(\alpha-\delta)T$ with $\delta\to0$ as $\epsilon\to0.$
Thus,
\begin{eqnarray}
\begin{split} \label{eq:gammalb}
\lim_{\be \to \8}& \lim_{T \to \8}   \frac{1}{\beta T}\ln \Ek[e^{\be H_T(x)}|N(T,X)=T]\\
\quad\ge&\lim_{\be \to \8} \lim_{T \to \8}   \frac{1}{\beta T}\ln \Ek[e^{\be H_T(x)},X\sim \gamma|N(T,X)=T]\\
\quad\ge &\lim_{\be \to \8} \lim_{T \to \8}   \frac{1}{\beta T}\ln e^{\be (\alpha-\delta)T}P_\kappa(X\sim \gamma|N(T,X)=T)\\
\quad\ge &\lim_{\be \to \8} \lim_{T \to \8}   \frac{1}{\beta T}\ln \lt(e^{\be (\alpha-\delta)T}\sqrt{T}\lt(\frac{\epsilon}{6e}\rt)^T\rt)\\
\quad\ge&(\alpha-\delta)
\end{split}
\end{eqnarray}
and letting $\epsilon\to0$ and therefore $\delta\to0$ we have established the claim (\ref{eq:stchely3}).
Then,  (\ref{eq:stchely3}) and (\ref{eq:gammalb}) imply (i), uniformly for $r \in [\epsilon,\epsilon^{-1}]$ for all 
$\epsilon \in (0,1]$. 

By  (\ref{eq:stchely1}) and (\ref{eq:gammaub}) we have
\begin{equation}
 \label{eq:stchely4}
\Gm( \beta, r) \leq \alpha \beta \sqrt{r},
\end{equation}
Combining (\ref{eq:valfreeenergy}) and (\ref{eq:stchely4}),
$$
\Pskb \leq 
\sup_{r \geq 0} \left\{ \alpha \be \sqrt{{r}}
- \textup{I}_\kappa(r) \right\}.
$$
For the converse direction, consider $r_m$ the (unique) maximizer
of  $\alpha \be \sqrt{{r}} - I_\kappa(r)$. Note that
$r_m \sim \frac{\alpha^2 \be^2}{4 \ln^2(\be^2/\kappa)}$ as $\be^2/\kappa \to \8$ and write
\begin{eqnarray*}
\Pskb 
&\geq&
\Gm(\be, r_m) -  \textup{I}_\kappa(r_m)\\
&=& {r_m} \Gm(r_m^{-1/2}\be, 1) -  \textup{I}_\kappa(r_m) \qquad({\rm by\ }
(\ref{eq:scalingGmseul}))
\\
&=& {r_m}^{-1/2}\be (1+ \epsilon(r_m^{-1/2}\be)) -  \textup{I}_\kappa(r_m) 
\qquad(\lim_{u \to 0} \epsilon(u)=0, {\rm see\ }
(\ref{eq:stchely3}))
\\
&=& (1+ \epsilon(r_m^{-1/2}\be)) \times  {\rm r.h.s.\ of\ }(\ref{eq:huit}).
\end{eqnarray*}
Thus, the first claim in (ii) is proved. The second one is clear.
We finish by (iii), which means that all maximizers are of the 
indicated order of magnitude. For $\epsilon >0$, using (\ref{eq:stchely4})
and the definition of $r_m$,
\begin{eqnarray*}
\sup \{ \Gm(\be,r) -  \textup{I}_\kappa(r): |r -r_m| \geq  r_m \epsilon\}
&\leq&
\sup \left\{ \alpha \be \sqrt{r} -  \textup{I}_\kappa(r); |r -r_m| \geq  r_m \epsilon
\right\}\\
&\leq& (1- \delta)
\sup \{ \alpha \be \sqrt{r} -  \textup{I}_\kappa(r); r \geq 0\}
\\
& \leq &
(1-\delta') \Pskb
\end{eqnarray*}
for some positive $\delta, \delta'$. 
\end{proof}
\section{Regularity of the favourite attributes} \label{sec:attributes}

In this section, we give the 

\begin{proof}[Proof of   Proposition \ref{prop:1}] We denote by $\|x\|=\max_{i \leq d}|x_i|$ the
supremum norm on the lattice. Let $0<a<b<\ln 2$. 

(i) Observe that provided $n$ is sufficiently large that $2^n/t>\kappa,$
\begin{eqnarray}
\begin{split}
P\left(\sup_{2^n\le \|x\|<2^{n+1}} f(x,t)>e^{-bn2^{n}}\right)&\le4^d2^{nd}\max_{2^n\le\|x\|<2^{n+1}}P(f(x,t)>e^{-bn2^{n}})\\
\le4^d 2^{nd}e^{bn2^{n}}& \max_{2^n\le\|x\|<2^{n+1}}E E_\kappa \left[ \exp \{ \beta H_t(X)\}\delta_x(X(t)) \right]\\
=4^d 2^{nd}e^{bn2^{n}}& \max_{2^n\le\|x\|<2^{n+1}}E_\kappa \left[E\left[ \exp \{\beta H_t(X)\}]\right]\delta_x(X(t))\right]\\
\le4^d 2^{nd}e^{bn2^{n}}&e^{\beta^2 t/2} P_\kappa(N(t,X)\ge 2^n)\\
\le4^d  2^{nd}e^{bn2^{n}}&e^{\beta^2t/2} e^{-2^{n}\log\left(2^n/\kappa t\right)+(2^n/t-\kappa)t}.
\end{split}
\end{eqnarray}
Thus,
\[\sum_{n=0}^\infty P(\sup_{2^n\le \|x\|<2^{n+1} }f(x,t)>e^{-bn2^{n}})<\infty,\]
and consequently, given almost any realization of $\mathcal{W},$ there is an $N_1$ such that 
\[f(x,t)\le e^{-bn2^{n}},\,\,\mbox{for}\, \|x\|\ge 2^n,\,\,n\ge N_1.\]
This implies that for some $N_1=N_1(\mathcal{W}),$
\[\sup_{x \in{{\ZZ^d}}}f(x,t)=\max_{x\in{{\ZZ^d}},\|x\|\le 2^{N_1}}f(x,t),\]
which implies the desired property.

(ii) Setting $\tau_x=\inf\{s\ge0:X(s)=x\},$ we have
\begin{align}
\begin{split}\label{favpath}
P(\sup_{2^n\le \|x\|<4^d 2^{n+1}} \sup_{ t\le T}f(x,t,T)>e^{-bn2^{n}})&\le4^d 2^{nd}\max_{2^n\le\|x\|<2^{n+1}}P(\sup_{ t\le T}f(x,t,T)>e^{-bn2^{n}})\\
\le4^d 2^{nd}e^{bn2^{n}}&\max_{2^n\le\|x\|<2^{n+1}}E\left[ \sup_{ t\le T}E_\kappa \left[ \exp \{ \beta H_T(X)\}\delta_x(X(t)) \right]\right]\\
\le4^d 2^{nd}e^{bn2^{n}}& \max_{2^n\le\|x\|<2^{n+1}}EE_\kappa \left[  \sup_{ t\le T}\exp \{ \beta H_T(X)\}\delta_x(X(t)) \right]\\
=4^d 2^{n}e^{bn2^{n}}& \max_{2^n\le\|x\|<2^{n+1}}EE_\kappa \left[ \exp \{ \beta H_T(X)\}; \tau_x\le T \right]\\
=4^d 2^{nd}e^{bn2^{n}} &\max_{2^n\le\|x\|<2^{n+1}}E_\kappa \left[E\left[ \exp \{ \beta H_T(X)\}]\right];\tau_x\le T\right]\\
\le4^d 2^{nd}e^{bn2^{n}}&e^{ \beta^2 T/2} P_\kappa(N( T,X)\ge 2^n)\\
\le 4^d 2^{nd}e^{bn2^{n}}&e^{\beta^2T/2} e^{-2^{n}\log\left(2^n/\kappa T\right)+(2^n/T-\kappa)T}.
\end{split}
\end{align}
In the same way as before, we now conclude that $f(x,t,T)$ exhibits superexponential decay in $x$ and as well that for some $N_2=N_2(\mathcal{W},T ),$
\[\sup_{x\in{{\ZZ^d}}}f(x,t,T)=\max_{x\in{{\ZZ^d}},\|x\|\le 2^{N_2}}f(x,t,T),\,0\le t\le T.\]

(iii) For $X \in {\mathcal D}_t$, define the time-reversed path $\hat X^{(t)} \in {\mathcal D}_t$ with
$\hat X^{(t)}(s)=X(t-s)$. Since the symmetric simple random walk is reversible, we have
\begin{eqnarray*}
f(x,t) &=& E_\kappa[\exp\{\beta H_t(X)\}  \delta_x(X(t)) ] \\
&=&  E_\kappa^x[\exp\{\beta H_t(\hat X^{(t)})\}  \delta_0(X(t)) ]. 
\end{eqnarray*}•
By (\ref{pam}), we obtain
$$
f(x,t)=    \delta_x(0)+ \kappa \int_0^t \Delta f(x,s) ds + \beta \int_0^t f(x,s) \circ dW_x(s).
$$
Since $f$ is positive, we see that $f(x,\cdot)$ has the same regularity as $W_x$.
\medskip

(iv) Define
\begin{equation} \label{eq:c}
\left\{
\begin{array}{l}
H_{t,T}^x(X)= \sum_{y \neq x} \int_t^T  \delta_y(X(s-t)) [dW_y(s)-dW_x(s)],\\
H_t^x(X)=H_{0,t}^x(X),\\
f^x(x,t,T)=f(z,t,T) \times  e^{-\beta W_x(T)}.
\end{array}•
\right.
\end{equation}•
By Markov property, we have
\begin{equation}\label{eq:b}
f^x(z,t,T)= f^x(z,t) \times g^x(z,t,T),
\end{equation}•
with
\begin{equation*} 
\left\{
\begin{array}{l}
f^x(z,t)= E_\kappa[ e^{\beta H_t^x(X)}  \delta_z(X(t)) ,\\
g^x(z,t,T)= E_\kappa^z[ e^{\beta H_{t,T}^x(X)}] .\\
\end{array}•
\right.
\end{equation*}•
By reversibility under $P_\kappa$, 
$$
f^x(z,t)=  E_\kappa^z[ e^{\beta H_t^x(\hat X^{(t)})}  \delta_0(X(t))].
$$
Similar to (\ref{pam}), we obtain for $z\neq x$,
$$
f^x(z,t)=    \delta_z(0)+ \kappa \int_0^t \Delta_z f^x(z,s) ds + \beta \int_0^t f^x(z,s) \circ d[W_z(s)-W_x(s)],
$$
though the last term vanishes for $z=x$,
\begin{equation}\label{eq:a}
f^x(x,t)=    \delta_x(0)+ \kappa \int_0^t \Delta_z f^x(x,s) ds .
\end{equation}•
By definition of $g^x(z,t,T)$, we have a similar identity,
$$
g^x(x,t,T)=1 + \kappa \int_t^T \Delta_z g^x(x,s,T) ds .
$$
Combining this with (\ref{eq:c}), (\ref{eq:b}), (\ref{eq:a}), we conclude that 
$t \mapsto f^x(z,t,T)$ is differentiable on $(0,T)$ with derivative 
$$
\frac{d}{dt} f(x,t,T) =  e^{\beta W_x(T)} \kappa \big( g^x(x,t,T) \Delta_z f^x(x,t) -  f^x(x,t) \Delta_z g^x(x,t,T) \big),
$$
which is continuous this interval, with probability 1.
\end{proof}

{\bf Acknowledgements:} FC thanks the Fields institute for hospitality during this work. 





\end{document}